\documentclass[10pt,twoside]{amsart}
 \usepackage{amsmath,amssymb,amsthm,amscd,latexsym,graphicx}
 \usepackage{epigraph}
 \usepackage[OT2,OT1]{fontenc}
 \usepackage[mathscr]{euscript}

\usepackage{enumitem}

\newcommand\scalemath[2]{\scalebox{#1}{\mbox{\ensuremath{\displaystyle #2}}}}

\usepackage{graphicx}
\usepackage[all]{xy}
\usepackage{hyperref}

 \font\caps=cmcsc10                    
 \font\Caps=cmcsc10 scaled \magstep1   

 \makeatletter
 \@specialpagefalse\headheight=8.5pt
 \makeatother

 \def\TSkip{\medskip}
 \newbox\TheTitle{\obeylines\gdef\GetTitle #1
 \ShortTitle  #2
 \SubTitle    #3
 \Author      #4
 \ShortAuthor #5
 \EndTitle
 {\setbox\TheTitle=\vbox{\baselineskip=20pt\let\par=\cr\obeylines%
 \halign{\centerline{\Caps##}\cr\noalign{\medskip}\cr#1\cr}}%
         \copy\TheTitle\TSkip\TSkip%
 \def\next{#2}\ifx\next\empty\gdef\STitle{#1}\else\gdef\STitle{#2}\fi%
 \def\next{#3}\ifx\next\empty%

 \else\setbox\TheTitle=\vbox{\baselineskip=20pt\let\par=\cr\obeylines%
     \halign{\centerline{\caps##} #3\cr}}\copy\TheTitle\TSkip\TSkip\fi%
 \centerline{\caps #4}\TSkip\TSkip%
 \def\next{#5}\ifx\next\empty\gdef\SAuthor{#4}\else\gdef\SAuthor{#5}\fi%
 \catcode'015=5}}\def\Title{\obeylines\GetTitle}

 \def\Abstract{\begingroup\narrower
     \parskip=\medskipamount\parindent=0pt{\caps Abstract. }}

 \long\def\MSC#1\EndMSC{\def\arg{#1}\ifx\arg\empty\relax\else
      {\par\narrower\noindent%
      2000 Mathematics Subject Classification: #1\par}\fi}

 \long\def\KEY#1\EndKEY{\def\arg{#1}\ifx\arg\empty\relax\else
         {\par\narrower\noindent Keywords and Phrases: #1\par}\fi\TSkip}

 \long\def\DATE#1\EndDATE{\def\arg{#1}\ifx\arg\empty\relax\else
         {\par\narrower\noindent \center{\textit{#1}}\par}\fi\TSkip\TSkip\TSkip}

 \font\bf= cmbx10 at 10pt

 \newcommand{\lra}{\longrightarrow}
 \newcommand{\lla}{\longleftarrow}

 \newcommand{\A}{\mathbb A}

 \newcommand{\F}{\mathbb{F}}

 \newcommand{\N}{\mathbb{N}}

 \newcommand{\Z}{\mathbb{Z}}

 \newcommand{\GL}{\mathrm{GL}}

  \newcommand{\End}{\mathrm{End}}
 \newcommand{\Ext}{\mathrm{Ext}}

 \newcommand{\Hom}{\mathrm{Hom}}
 
  \newcommand{\Ver}{\mathrm{Ver}}
   \newcommand{\Frob}{\mathrm{Frob}}

 \newcommand{\M}{\mathcal{M}}

\newcommand{\Id}{\mathrm{Id}}

\newcommand{\Ker}{\mathrm{Ker}}
 
 \newcommand{\Spec}{\mathrm{Spec}}

\renewcommand{\M}{\mathcal{M}}
\newcommand{\W}{\mathbf{W}}
\newcommand{\E}{\mathcal{E}}

\newcommand{\YEExt}{\mathbf{YExt}}
\newcommand{\YExt}{\mathrm{YExt}}

 \theoremstyle{plain}
 \newtheorem{thm}{Theorem}[section]
 \newtheorem{defi}[thm]{Definition}

 \newtheorem{prop}[thm]{Proposition}

 \newtheorem{lem}[thm]{Lemma}
 \newtheorem{coro}[thm]{Corollary}

 \theoremstyle{remark}
 \newtheorem{rem}[thm]{Remark}

    \newtheorem{notation}[thm]{Notation}
 \newenvironment{dem}{{\bf Proof.}}{\hfill$\square$}

\date{December 2018}

\newif\ifquoteopen
\catcode`\"=\active 
\DeclareRobustCommand*{"}{%
   \ifquoteopen
     \quoteopenfalse ''%
   \else
     \quoteopentrue ``%
   \fi
}
\begin{document}


 \Title
Lifting low-dimensional local systems
 \ShortTitle

\SubTitle
 \Author
 Charles De Clercq\footnote{Partially supported by French ministries of Foreign Affairs and of Education and Research (PHC Sakura- New Directions in Arakelov Geometry).} and Mathieu Florence
 \ShortAuthor
 \EndTitle

\address{Charles De Clercq, Equipe Topologie Alg\'ebrique, Laboratoire Analyse, G\'eom\'etrie et Applications, Universit\'e Sorbonne Paris Nord, 93430 Villetaneuse.}
\address{Mathieu Florence, Equipe de Topologie et G\'eom\'etrie Alg\'ebriques, Institut de Math\'ematiques de Jussieu,  Sorbonne Universit\'e, 4, place Jussieu, 75005 Paris. }

\Abstract
Let $k$ be a field of characteristic $p>0$. Denote by  $\W_r(k)$ the ring of truntacted Witt vectors of length $r \geq 2$, built out of $k$. In this text,  we consider the following question, depending on a given profinite group $G$.

$Q(G)$: Does every (continuous) representation  $G\lra \GL_d(k)$ lift to a representation $G\lra \GL_d(\W_r(k))$?

We work in the class of cyclotomic pairs (Definition \ref{deficyclopair}), first introduced in \cite{DCF} under the name  "smooth profinite groups". Using Grothendieck-Hilbert' theorem 90,  we show that the algebraic fundamental groups of the following schemes are cyclotomic: spectra of semilocal rings over  $\Z[\frac 1 p]$, smooth curves over algebraically closed fields, and affine schemes over $\F_p$. In particular,  absolute Galois groups of  fields  fit into this class. We then give a positive partial answer to $Q(G)$, for   a cyclotomic profinite group $G$: the answer is positive, when $d=2$ and $r=2$. When $d=2$ and $r=\infty$, we  show that any  $2$-dimensional representation of  $G$ \textit{stably} lifts to  a representation over $\W(k)$: see Theorem \ref{t1}. \\When $p=2$ and $k=\F_2$, we prove the same results, up  to dimension $d=4$. \\We then give a concrete application to algebraic geometry: we prove that local systems of low dimension lift Zariski-locally (Corollary \ref{tf}).

\tableofcontents

\section{Introduction}

Let $k$ be a field of characteristic $p$ and let $G$ be a profinite group. This paper deals with the deformation theory  (more accurately, the liftability)  of continuous representations  $$\rho:G\longrightarrow \GL_d(k),$$
 ultimately with coefficients in the ring of Witt vectors $\W(k)$. A fundamental instance is given by Galois representations. \\Existence of such lifts has been extensively investigated, in the case of absolute Galois groups of local and global fields. In \cite{K}, Khare proves the existence of lifts to  $\W(k)$, for $2$-dimensional reducible representations, in the case where $G$ is the absolute Galois group of a number field $F$, and when $k$ is a finite field. As noticed by Serre, the proof actually works for any field $F$. If $G$ is the absolute Galois group of $\mathbb{Q}$, and under mild assumptions, such lifts  exist more generally, by the work of Ramakrishna \cite{R}. Some time after the present text was released, Khare and Larsen (\cite{KL}) proved  that the answer to $Q(G)$ is positive, when $G$ is the absolute Galois group of a non-archimedean local field, or a global field, when $k=\F_p$ for odd $p$, $d\leq 3$ and $r=2$.\\
A class of profinite groups, whose mod $p$ representations are likely to lift mod $p^2$,  was first introduced in \cite{DCF} under the name \emph{smooth profinite groups}. Due to the recent progress made in the series of papers \cite{DCF1}, \cite{F2} and   \cite{DCF3}, it is now clear that one should distinguish between the notions of \emph{smooth profinite groups} and of \emph{cyclotomic pairs}. In the present paper, we focus on cyclotomic pairs. Loosely speaking, a cyclotomic pair consists of a profinite group, equipped with a so-called cyclotomic module, which will play the role of the cyclotomic character in Galois cohomology- see Definition \ref{deficyclopair} . We say that a profinite group is cyclotomic, if it fits into a cyclotomic pair.\\

The main contribution of this paper, is to give important examples of cyclotomic profinite groups among algebraic fundamental groups, using Kummer and Artin-Schreier theory.\\
More precisely, in Propositions \ref{semiloccyclotomic}, \ref{charpcyclotomic},  \ref{curvecyclotomic} and \ref{vastcyclotomic},  we show that the fundamental group $\pi_1(S,\overline s)$, of a given scheme $S$ at a geometric point $\overline s$, is cyclotomic in each of the following cases.
\begin{enumerate}[label=\alph*)]
    \item $S$ is  a semilocal $\Z[\frac 1 p]$-scheme;
    
    \item $S$ is an affine $\F_p$-scheme;
    
    \item $S$ is a smooth curve, over an algebraically closed field.
     \item More generally, $S$ is a smooth projective variety  over an algebraically closed field, such that for every finite \'etale cover $U/S$, the N\'eron-Severi group of $U$ is torsion-free.
\end{enumerate}Note that, in all four cases, the cyclotomic module is taken to be the Tate module of roots of unity $\Z_p(1)$ when $p$ is invertible on $S$, or the trivial module $\Z_p$ when $p$ vanishes on $S$. In particular, absolute Galois groups of fields are cyclotomic profinite groups.

Our next result is Theorem \ref{t1}: continuous representations  of  dimension $2$ of a cyclotomic profinite group (e.g. of  type a), b), c) or d) above), with values in an \textit{arbitrary} field $k$ of characteristic $p>0$ (possibly infinite), lift to $p^2$-torsion coefficients. They also \textit{stably} lift to arbitrary torsion (see Definition \ref{stablelift} for the notion of stable lifting). \\For $p=2$ and $k= \F_2$, we prove the same results, for representations of dimension \textit{up to $4$}. After this paper was first drafted, the recent text \cite{F2} was released, in which it is proved that mod $p$ representations of a cyclotomic profinite group lift mod $p^2$- in all dimensions $d$. The proof involves a delicate new technology. Our theorem \ref{t1} here is the particular case $d=2$; it is much easier to read.\\

The paper is structured as follows. In section \ref{SectionWitt}, we recall the machinery  of Witt vectors and of Yoneda extensions, which is  a convenient computational tool in our proofs. We give precise definitions of what is meant by "lifting" in section \ref{SectionLift}. In section \ref{SectionCyclo}, we recall the notion of a cyclotomic pair.  In the remaining sections, we prove the lifting theorem  and deal with its applications.\\



\section{Modules over Witt vectors and Yoneda extensions}\label{SectionWitt}

Fixing a field $k$ of characteristic $p>0$, we consider the ring $\W(k)$ of Witt vectors built out of $k$. Recall that if $k$ is perfect, $\W(k)$ is the unique complete discrete valuation ring of characteristic $0$ whose uniformizer is $p$ and residue field is $k$. We shall also consider the truncated Witt vectors of size $r\geq 1$, defined by the quotient $$\W_r(k):=\W(k)/\mathrm{Ver}^r(\W(k)),$$ where $\mathrm{Ver}:\W(k) \lra \W(k) $ denotes the Verschiebung endomorphism. We set $\W_{\infty}(k):=\W(k)$. Note that if $k$ is perfect, we have $\W_r(k)=\W(k)/p^r\W(k)$.

\begin{defi}
Let $r\in \mathbb{N}^{\ast}\cup \{\infty\}$. Let $M$ be a $\W_r(k)$-module of  finite type.\\ We endow it with the topology having the submodules $ \Ver^i(\W(k)) M$, $i\in \N$, as a basis for open neighborhoods of $0$. This is simply the discrete topology if  $r < \infty$.\\
Note also that,  if $k$ is perfect, this defines the $p$-adic topology on $M$.
\end{defi}
\begin{defi}\label{wkgmod}
Let $G$ be a profinite group, and let $r\in \mathbb{N}^{\ast}\cup \{\infty\}$.

A $(\W_r(k),G)$-module is a $\W_r(k)$-module of finite type, endowed with a continuous $\W_r(k)$-linear action of $G$.\\In particular, for $r=1$, a $(k,G)$-module is a finite-dimensional $k$-vector space endowed with an action of $G$, that factors through an open subgroup of $G$.\\
If $V$ is a  $(\W_r(k),G)$-module, we set \[ V^\vee:= \Hom_{\W_r(k)}(V, \W_r(k)).\] We refer to $V^\vee$ as the Pontryagin dual of $V$. \\If $k$ is perfect, then Pontryagin duality is perfect, in the sense that the natural arrow $$ V \lra V ^{\vee \vee} $$ is an isomorphism.

\end{defi}
If $G$ is a profinite group and $V$ is a $G$-module, we will denote by $H^n(G,V)$ the $n$-th cohomology group of $G$, with values in $V$. If $V$ is equipped with the discrete topology, it is taken in the sense of \cite{Se}. Otherwise (e.g. when $V=\Z_p(n)$), it is in the sense of Tate's continuous cohomology. In the context of the present paper, laying too much stress on continuity issues would, we believe, be smoke and mirrors.\\  The abelian categories $\mathcal{M}(\W_r(k),G)$ (resp. $\mathcal{M}(k,G)$) of $(\W_r(k),G)$-modules (resp. $(k,G)$-modules) are monoidal through the tensor product. For any positive integer $n$ and $A,B \in \M( \W_r(k),G)$, one can then define the notion of Yoneda $n$-extensions of $B$ by $A$, as follows. First, define $\YExt_{( \W_r(k),G)}^0(B,A)$ as $\Hom_{(\W_r(k),G)}(B,A)$.

A $n$-extension of $B$ by $A$ is an exact sequence of $(\W_r(k),G)$-modules $$\E: 0 \lra A \lra A_1 \lra \ldots \lra A_{n}\lra B \lra 0.$$ Setting morphisms $\E_1\lra \E_2$ between two $n$-extensions of $B$ by $A$ to be morphisms of complexes for which the induced morphisms between $A$ and $B$ are the identity maps, we get the category $\YEExt_{( \W_r(k),G)}^n(B,A)$ of Yoneda extensions, of $B$ by  $A$ and of size $n$. It is additive through the Baer sum.

Any morphism of $(\W_r(k),G)$-modules $f:A \lra A'$ (resp. $g:B' \lra B$) induces a pushforward functor $$f_*:  \YEExt_{( \W_r(k),G)}^n(B,A) \lra \YEExt_{( \W_r(k),G)}^n(B,A'), $$ resp. a pullback functor $$g^*: \YEExt_{( \W_r(k),G)}^n(B,A) \lra \YEExt_{( \W_r(k),G)}^n(B',A).$$ Those functors commute, in the sense that $f_* g^*$ and $g^* f_*$ are canonically isomorphic.

Let us say that two Yoneda extensions $\E_1$ and $\E_2$ in $\YEExt_{(\W_r(k),G)}^n(B,A)$ are linked if there exists a third extension $\mathcal{E}\in \YEExt_{(\W_r(k),G)}^n(B,A)$ and morphisms of $n$-extensions
$$\E_1 \lla \E  \lra \E_2 .$$
In our setting, this indeed defines an equivalence relation between elements of $\YEExt_{(\W_r(k),G)}^n(B,A)$, compatible with  Baer sum.

\begin{defi}
 We denote by $\YExt_{( \W_r(k),G)}^n(B,A)$   the Abelian group of  equivalence classes of Yoneda $n$-extensions, in the category  $\YEExt_{( \W_r(k),G)}^n(B,A)$.
\end{defi}

\begin{prop}\label{extcoh}
Let $r\in \mathbb{N}^{\ast}\cup \{\infty\}$, and let $V$ be  a $(\W_r(k),G)$-module. Then, for any $n \geq 0$, there is a  canonical isomorphism $$ \YExt_{(\W_r(k),G)}^n(\W_r(k),V) \simeq H^n(G,V).$$
\end{prop}

\begin{dem}
	Let us first deal with the case where $G$ and $r$ are  finite.\\
	The group $H^n(G,V)$  is the $n$-th derived functor of the functor $$V \mapsto V^G=\Hom_{(\W_r(k),G)}(\W_r(k),V).$$ Thus, it is nothing but the usual $\Ext$ group $\Ext_{(\W_r(k),G)}^n(\W_r(k),V) $, computed using injective resolutions.  But, for any Abelian category with enough injectives, the derived $\Ext$'s coincide with the Yoneda $\YExt$'s (\cite{Ve}, Ch. III,   Par. 3).\\
	The general case  follows from a  classical limit argument, over the finite quotients of $G$. 
\end{dem}
\begin{lem}\label{cohhom}
	Let $r\in \mathbb{N}^{\ast}\cup \{\infty\}$ and let $A,B$ be  two $(\W_r(k),G)$-modules, $B$ assumed to be free as a $\W_r(k)$-module. Then, for any $n \geq 0$, there is a  canonical isomorphism $$ \YExt_{(\W_r(k),G)}^n(B,A) \stackrel \sim \lra \YExt_{(\W_r(k),G)}^n(\W_r(k),\Hom_{\W_r(k)}(B,A)) .$$
\end{lem}

\begin{dem}
	Considering the Pontryagin dual $B^{\vee}=\Hom_{\W_r(k)}(B,\W_r(k))$, we have a canonical isomorphism $$B^{\vee} \otimes A \stackrel \sim \lra \Hom_{\W_r(k)}(B,A).$$
	The exact functor $ B^{\vee} \otimes .$ yields a functor $$T: \YEExt_{(\W_r(k),G)}^n(B,A) \lra \YEExt_{(\W_r(k),G)}^n(B^{\vee} \otimes B,B^{\vee} \otimes  A)$$ which maps a Yoneda $n$-extension $$\E: 0 \lra A \lra A_1 \lra \ldots \lra A_{n}\lra B \lra 0$$ to the Yoneda $n$-extension $$ B^{\vee} \otimes\E: 0 \lra  B^{\vee} \otimes A \lra  B^{\vee} \otimes A_1 \lra \ldots \lra  B^{\vee} \otimes A_{n}\lra  B^{\vee} \otimes B \lra 0.$$ But the $G$-equivariant  map $$I: \W_r(k) \lra B^{\vee} \otimes B =\End_{\W_r(k)}(B),$$ $$ \lambda \mapsto \lambda \Id$$ gives a pullback functor $$\scalemath{0.95}{I^*: \YEExt_{(\W_r(k),G)}^n(B^{\vee} \otimes B,B^{\vee} \otimes  A) \lra \YEExt_{(\W_r(k),G)}^n(\W_r(k), \Hom_{\W_r(k)}(B,A)),}$$
	
	and the composite $$ I^* \circ T:   \YEExt_{(\W_r(k),G)}^n(B,A) \lra \YEExt_{(\W_r(k),G)}^n(\W_r(k), \Hom_{\W_r(k)}(B,A))$$ gives, by passing to isomorphism classes of objects, a group  homomorphism $$\Phi: \YExt_{(\W_r(k),G)}^n(B,A) \lra \YExt_{(\W_r(k),G)}^n(\W_r(k), \Hom_{\W_r(k)}(B,A)),$$ which is the desired isomorphism. Its inverse can be constructed in a similar fashion, as follows. Given a Yoneda $n$-extension of $(\W_r(k),G)$-modules $$\mathcal F: 0 \lra B^\vee \otimes A \lra F_1 \lra \ldots \lra F_{n}\lra \W_r(k) \lra 0,$$ form the tensor product $$\mathcal F \otimes B: 0 \lra B^\vee \otimes B  \otimes A \lra F_1 \otimes B \lra \ldots \lra F_{n} \otimes B\lra B \lra 0.$$ Applying pushforward w.r.t. the trace  $$  B^\vee \otimes B  \otimes A  \lra A, $$ $$ \phi \otimes b \otimes a \mapsto \phi(b) a,$$ we get  an $n$-extension $$\mathcal E: 0 \lra A \lra A_1 \lra \ldots \lra A_n \lra B \lra 0.$$ Passing to isomorphism classes of extensions, we get an arrow $$\Psi:  \YExt_{(\W_r(k),G)}^n(\W_r(k), \Hom_{\W_r(k)}(B,A)) \lra \YExt_{(\W_r(k),G)}^n(B,A).$$ Checking that $\Phi$ and $\Psi$ are mutual inverses is left as an exercise for the reader- in the spirit of Morita equivalence.
\end{dem}
\section{Lifting and stable lifting}\label{SectionLift}

The purpose of this section is to give a precise meaning to "lifting representations". Here $G$ is a profinite group, and $k$ is any field of characteristic $p$.

\begin{defi}(Lifting, stable lifting).\\ \label{stablelift}
Let $1 \leq r \leq s$ be integers. \\Let $V_r$ be a $(\W_r(k),G)$-module, free as a $\W_r(k)$-module. We say that $V$ lifts to $p^s$-torsion coefficients, if there exists a $(\W_s(k),G)$-module $V_s$, free as a $\W_s(k)$-module, such that the $(\W_r(k),G)$-modules $V_r$ and $V_s \otimes_{\W_s(k)} \W_r(k)$ are isomorphic.\\
We say that $V_r$ stably lifts to $p^s$-torsion coefficients, if there exists an open subgroup $G_0 \subset G,$ of prime-to-$p$ index, such that, as a $(\W_r(k),G_0)$-module, $V_r$ lifts to $p^s$-torsion coefficients.

\end{defi}

The terminology "stable" is motivated by the following Lemma.
\begin{lem}
Let $1 \leq r \leq s$ be integers.
Let $V_r$ be a $(\W_r(k),G)$-module, free as a $\W_r(k)$-module. Assume that $V_r$ stably lifts to $p^s$-torsion coefficients. Then, there exists a $(\W_r(k),G)$-module $W_r$,   free as a $\W_r(k)$-module, such that $V_r \oplus W_r$ lifts to $p^s$-torsion coefficients.

\end{lem}

\begin{dem}
Let $G_0 \subset G$ be an open subgroup, of prime-to-$p$
 index, such that the $(\W_r(k),G_0)$-module $V_r$ lifts to $p^s$-torsion coefficients. Let $V_s$ be a $(\W_s(k),G_0)$-module,  free as a $\W_s(k)$-module, lifting $V_r$.\\ Denote by $V_r^{(G/G_0)}$ the product of copies of  $V_r$, indexed by the finite set $G/G_0$.  It is a $(\W_r(k),G)$-module in a natural way, canonically isomorphic to $\mathrm{Ind}_{G_0}^G(\mathrm{Res}_{G_0}^G(V_r))$.\\
Consider the morphisms of $(\W_r(k),G)$-modules $$V_r\stackrel{i}\lra V_r^{(G/G_0)} \stackrel{N}\lra V_r,$$ where $i$ is the diagonal embedding, and $N$ is the norm  $$ \begin{array}{lcll}
N :& V_r^{(G/G_0)} &\longrightarrow &V_r\\
        & (v_c)_{c\in G/G_0} &\longmapsto&\sum_{c\in G/G_0}v_c.
\end{array}$$ 
The composite $N\circ i$ is  multiplication by the index of $G_0$ in $G$, which is prime to $p$. The $(\W_r(k),G)$-module $V_r$ is thus a direct summand of $V_r^{(G/G_0)}$, with complement $W_r:=\Ker(N)$. But the $(\W_r(k),G)$-module module $V_r^{(G/G_0)}$  admits the induced module $$\mathrm{Ind}_{G_0}^G{V_s}=V_s \otimes_{\W_{s}(k)[G_0]}\W_{s}(k)[G]$$ as a lift to $p^{s}$-torsion coefficients. The claim follows.\end{dem}

\begin{rem}
 In the previous Lemma, once $G_0$ is known, $W_r$ is pretty much explicit.
\end{rem}

Lifting from mod $p^r$ to mod $p^{r+1}$ is an "abelian" question: there is no difference between lifting and stable lifting, as illustrated by the following Lemma.
\begin{lem}\label{ablift}
Let $1 \leq r$ be an integer. Let $V_r$ be a $(\W_r(k),G)$-module, free as a $\W_r(k)$-module. Assume that there is another $(\W_{r}(k),G)$-module $W_r$, such that $V_r\oplus W_r$ lifts to $p^{r+1}$-torsion coefficients.\\Then, $V_r$ itself lifts to $p^{r+1}$-torsion coefficients.
\end{lem}
\begin{proof}
We give a constructive proof, avoiding the use of cohomological obstructions.\\
Denote by $V_1$ (resp. $W_1$) the reduction of $V_r$ (resp. $W_r$) to a $(k,G)$-module. If $M$ is a $\W(k)$-module, denote by $$M^{(i)}:=M \otimes_{\Frob^i} \W(k)$$ its $i$-th Frobenius twist.
By assumption, there is a free $(\W_{r+1}(k),G)$-module $Z_{r+1}$ and a short exact sequence  of $(\W_{r+1}(k),G)$-modules $$\mathcal{E}:0\lra V_1^{(r)}\oplus W_1^{(r)} = \Ver^r(\W_{r+1}(k)) Z_{r+1} \lra {Z_{r+1}} \lra   V_r\oplus W_r  \lra 0.$$
Denote by $i:V_r\lra V_r\oplus W_r$ and $\pi:W_1^{(r)}\oplus V_1^{(r)}\lra V_1^{(r)}$ the natural inclusion and projection. Form the extension of $(\W_{r+1}(k),G)$-modules $$\pi^{\ast}i_{\ast}(\mathcal{E}):0\lra V_1^{(r)}\lra V_{r+1} \lra V_r  \lra 0,$$ which serves as the definition of $ V_{r+1}$. Recall that the extensions $\pi^{\ast}i_{\ast}(\mathcal{E})$ and $i_{\ast}\pi^{\ast}(\mathcal{E})$ are canonically isomorphic, so that this construction does not depend on the order in which the pullback and the pushforward are applied. We claim that $ V_{r+1}$
is a lift of $V_r$ to $p^{r+1}$-torsion coefficients. To see why, it suffices to justify that $ V_{r+1}$ is free, as a $\W_{r+1}(k)$-module. We may thus dismiss the action of $G$. Picking bases, we then get that $V_r$ (resp. $W_r$, $Z_{r+1}$) is isomorphic to $\W_{r}(k)^m$ (resp.  $\W_{r}(k)^n$,  $\W_{r+1}(k)^m \oplus  \W_{r+1}(k)^n$), and that $\mathcal E$ is isomorphic to  $$0\lra k^m \oplus k^n  \lra \W_{r+1}(k)^m \oplus  \W_{r+1}(k)^n \lra   \W_{r}(k)^m \oplus  \W_{r}(k)^n \lra 0,$$ which is the direct sum of the extensions $$\mathcal F_i:=0\lra k^i  \lra \W_{r+1}(k)^i \lra   \W_{r}(k)^i \lra 0,$$ for $i=m,n$. Applying $i_{\ast}\pi^{\ast}$ to $\mathcal E$ yields  $\mathcal F_m$ as a result. The claim is proved. \end{proof}
\section{Cyclotomic modules and cyclotomic profinite groups}\label{SectionCyclo}

From now on, we fix a field $k$ of characteristic $p$ and a profinite group $G$.

In this section, we recall the notion of cyclotomic pair from \cite{DCF1}, and provide important examples.

\begin{notation}
Given two positive integers $s\leq r$ in $\mathbb{N}^{\ast}\cup \{\infty\}$ and a $(\W_r(k),G)$-module $M_r$, we put $$M_s:=M_r\otimes_{\W_r(k)}\W_s(k).$$
\end{notation} 

\begin{defi}\label{definsurj}
Let $r\in \mathbb{N}^{\ast}\cup \{\infty\}$ and $n \geq 1$ be an integer. Let $$f: M_r \lra N_r$$ be a  morphism of $(\W_r(k),G)$-modules. We say that $f$ is $n$-surjective if, for every open subgroup $H \subset G$,  the map $$f_*:H^n(H,M_r) \lra H^n(H,N_r)$$ is surjective.
\end{defi}

\begin{defi}\label{deficyclopair}
Let $n\geq 1$ be an integer and $e\in \mathbb{N}^{\ast}\cup \{\infty\}$. Let $\mathcal{T}_{e+1}$ be a  $(\W_{e+1}(k),G)$-module, free of rank $1$ as a $\W_{e+1}(k)$-module. \\Assume that the quotient $$ \mathcal{T}_{e+1}^{ \otimes^n_{\W_{e+1}(k)}} \rightarrow  \mathcal{T}_{1}^{ \otimes^n_{k}}$$ is $n$-surjective.\\
We then say that the pair $(G, \mathcal{T})$ is $(n,e)$-cyclotomic, and that the profinite group $G$ is $(n,e)$-cyclotomic (relatively to $k$).
\end{defi}

Let $\mathcal{T}_{e+1}$ be a $(n,e)$-cyclotomic module. For $i$ a non negative integer,  we  put $$\W_{e+1}(k)(i):=\mathcal{T}_{e+1}^{ \otimes^i_{\W_{e+1}(k)}}$$ and $$\W_{e+1}(k)(-i):=\W_{e+1}(k)(i) ^\vee.$$  For a $\W_{e+1}(k)$-module $M$, we put $$M(i)=  M\otimes_{\W_{e+1}(k)} \W_{e+1}(k)(i).$$

A cyclotomic module of depth $e$ is given by a continuous character $$\chi: G \lra \W_{e+1}(k)^\times,$$ 
and provides an analogue of the cyclotomic character in Galois theory. This allows to freely to freely mimic Kummer theory, in the framework of cyclotomic pairs. 


\begin{rem}
Let $(G, \W_{e+1}(k)(1))$ be a cyclotomic pair. If $H \subset G$ is a closed subgroup, then $(H, \W_{e+1}(k)(1))$ is a cyclotomic pair as well. This follows, by a standard limit argument, from the (obvious) case where $H \subset G$ is open.\\

\end{rem}
\begin{rem}\label{NotCyclo}
For $p$ odd, there is no non-trivial finite $(1,1)$-cyclotomic $p$-group. For $p=2$, the  only  non-trivial $(1,1)$-cyclotomic finite $2$-group  is $\Z/2\Z$. For each $e\in \mathbb{N}^{\ast}\cup \{\infty\}$, it fits into the unique $(1,e)$-cyclotomic  pair  $(\Z/ 2 \Z, \W_{e+1}(k)(1))$, where the non-trivial element of $\Z/2 \Z$ acts by multiplication by $-1$. This result is a variation around Emil Artin's theorem: the only non-trivial finite group that  occurs as an absolute Galois group is $\Z/ 2 \Z$.  See \cite{DCF}, Exercise 14.27, or \cite{QW}, Proposition 6.1.
\end{rem}
We now provide a supply of cyclotomic pairs $(G,\mathcal T)$ arising from geometry, i.e. where $G$ is the algebraic fundamental group of a scheme. Note that, by \cite[Proposition 15]{Se3}, it is known that every finite group occurs as the algebraic fundamental group of a smooth complex variety. Using Remark \ref{NotCyclo}, we see that not every algebraic fundamental group is cyclotomic profinite.

\begin{lem}\label{pdimcoh} Let $n \geq 1$ be an integer. Assume that the profinite group $G$ is of cohomological $p$-dimension at most $1$.  Let $\W_{e+1}(k)(1)$ be a $(\W_{e+1}(k),G)$-module, free  of rank $1$ as a $\W_{e+1}(k)$-module. Then, the pair $(G, \W_{e+1}(k)(1))$ is $(n,e)$-cyclotomic.
\end{lem}

\begin{dem}
Let $c_1 \in H^n(G, \W_{1}(k)(n))$ be a cohomology class. Using the exact sequence \[ 0 \lra \W_{e}(k)(n) \lra \W_{e+1}(k)(n) \lra \W_{1}(k)(n) \lra 0,\] we see that the obstruction to lifting $c_1$ to a class $c_{e+1} \in H^n(G, \W_{e+1}(k)(n))$ lies in  $H^{n+1}(G, \W_{e}(k)(n))$. This group vanishes since $G$ is  of cohomological $p$-dimension at most $1$. The claim is proved.
\end{dem}

\begin{prop}\label{toFp}
Assume that $(G, \W_{e+1}(k)(1))$ is an $(n,e)$-cyclotomic pair.\\ Then for any surjection $\pi:M\lra N$ of $\W_{e+1}(k)$-modules (with trivial $G$-action), the induced morphism $$\pi(n):M(n)\lra N(n)$$ is $n$-surjective.
\end{prop}

\begin{proof}
By a limit argument, we can assume that $e$ is finite. We then proceed by induction, on the lowest integer $m \geq 1$ such that $N$ is a $\W_m(k)$-module. If $m=1$, then $N$ is a $k$-vector space. Pick a $k$-basis $\mathcal B$ for $N$. Consider the natural surjection \[F: \W_{e+1}(k)^{(\mathcal B)} \lra k^{(\mathcal B)} \simeq N .\] There exists a $\W_{e+1}(k)$-linear map $ \rho: \W_{e+1}(k)^{(\mathcal B)} \lra M,$ such that $\pi \circ \rho=F. $ Since $F(n)$ is $n$-surjective by definition of a cyclotomic module (combined to the fact that taking cohomology commutes with direct sums), we indeed conclude that $\pi(n)$ is $n$-surjective as well.\\
In general, denote by $\mathcal M:= \Ver(\W_{e+1}(k)) $ the maximal ideal of $\W_{e+1}(k)$. Consider the composite \[ M \stackrel \pi \lra N  \stackrel q \lra N/ \mathcal M N,\] where $q$ is the natural quotient. By what precedes, $(q \circ \pi)(n)$ is $n$-surjective. It suffices to prove that $\pi'(n)$ is $n$-surjective, where  \[\pi': \mathcal M M \lra \mathcal M N \] denotes the map induced by $\pi$, by a diagram chase over 
$$\xymatrix{
    0\ar[r]&\mathcal M M\ar[d]^-{\pi'}\ar[r]& M\ar[d]^-{\pi}\ar[r] & M/\mathcal M M\ar[d]\ar[r] & 0 \\
    0\ar[r]&\mathcal M N\ar[r] & N\ar[r]^-q & N/\mathcal M N\ar[r] & 0 }$$
But as $\mathcal M N$ is a $\W_{m-1}(k)$-module, induction applies.
\end{proof}

\begin{coro}\label{extscal}
Let $l/k$ be a field extension. Let $n\geq 0$ be an integer, and let $e\in\mathbb{N}^{\ast}\cup\{\infty\}$. Let $(G, \W_{e+1}(k)(1))$ be an $(n,e)$-cyclotomic pair. Set $$ \W_{e+1}(l)(1):=\W_{e+1}(k)(1)\otimes_{\W_{e+1}(k)}\W_{e+1}(l).$$  Then, the pair $(G, \W_{e+1}(l)(1))$ is  $(n,e)$-cyclotomic, relatively to $l$. \\
In short: cyclotomic pairs are preserved under field extensions of $k$.
\end{coro}

\begin{proof}
The $(\W_{e+1}(l),G)$-module $\W_{e+1}(l)(1)$ is free of rank $1$. As a morphism of $\W_{e+1}(k)$-modules, the map $\W_{e+1}(l)\lra l$ is surjective. It remains to apply Proposition \ref{toFp}.
\end{proof}

Hilbert 90 theorem implies that the absolute Galois group $G$ of a field $F$ of characteristic $\neq p$, together with its Tate module $\Z_p(1)$, form a  $(1,\infty)$-cyclotomic pair. This elementary fact was discussed in \cite[Proposition 14.19]{DCF}, which also includes other examples of  (not necessarily absolute) Galois groups.\\
We now provide more geometric examples of cyclotomic pairs.

\begin{prop}\label{semiloccyclotomic}
Let $A$ be a semilocal $\Z[\frac 1 p]$-algebra. Denote by $G$ the \'etale fundamental group of $S:=\Spec(A)$ (at a given geometric point) and by  $\mathbb{Z}_p(1)$ its usual Tate module. Then, the pair $(G, \mathbb{Z}_p(1))$ is $(1,\infty)$-cyclotomic (over $k=\F_p$).
\end{prop}

\begin{proof}
May assume that  the semilocal ring $A$ is connected. We work on the small \'etale site over $S$. \\

An open subgroup of $G$ corresponds to the fundamental group $G_U$ of a finite \'etale cover $U\lra S$. Consider for $s\geq 1$ the diagram of \'etale sheaves
$$
\xymatrix{
0\ar[r]&\mu_{p,U}\ar[r]\ar@{=}[d]& \mu_{p^{s+1},U}\ar[r]^{\phi}\ar[d]&\mu_{p^s,U}\ar[r]\ar[d]&0\\
0\ar[r]&\mu_{p,U}\ar[r]& \mathbb{G}_{m,U}\ar[r]^{\phi'}&\mathbb{G}_{m,U}\ar[r]& 0
}
$$    

where $\phi$ and $\phi'$ denote the $p$-power maps. As $U$ is the spectrum of a semilocal ring, its Picard group is trivial by Grothendieck-Hilbert's theorem 90, and $\phi'$ certainly(!) induces a surjection $$H^1_{\acute{e}t}(U,\mathbb{G}_m)\lra H^1_{\acute{e}t}(U,\mathbb{G}_m).$$ A simple diagram chase then implies that $\phi$ also induces a surjection
$$H^1(G_U,\mathbb{Z}/p^{s+1}(1))\simeq H^1_{\acute{e}t}(U,\mu_{p^{s+1},U})\lra H^1_{\acute{e}t}(U,\mu_{p^s,U})\simeq H^1(G_U,\mathbb{Z}/p^{s}(1)).$$

\end{proof}


\begin{prop}[{\cite[Proposition 1.6]{Gi}}]\label{charpcyclotomic}
Let  $A$ be a commutative ring of characteristic $p$. Denote by $G$ the \'etale fundamental group of $S:=\Spec(A)$ (at a given geometric point). Then $G$ is of $p$-cohomological dimension $\leq 1$. Thus Lemma \ref{pdimcoh} applies: for any $n \geq 1$ and for any $\W(k)(1)$, the pair   $(G, \W(k)(1))$ is  $(n,\infty)$-cyclotomic.
\end{prop}

\begin{proof}(sketch; see  \cite[Proposition 1.6]{Gi} for details)\\
As before, we can assume that $k$ is $\mathbb{F}_p$,  and we work in the small \'etale site over $S$.  Consider the Artin-Schreier sequence $$0\lra \mathbb{Z}/p\mathbb{Z}\lra \mathbb{G}_a \stackrel {\Frob-\Id} \lra \mathbb{G}_a \lra 0.$$
By Grothendieck-Hilbert 90 for $\mathbb{G}_a$, combined with the vanishing of coherent cohomology over an affine base, we know that $H^1_{\acute{e}t}(S,\mathbb{G}_a)=H^2_{\acute{e}t}(S,\mathbb{G}_a)=0$. Considering the associated long sequence in \'etale cohomology, we get $H^2_{\acute{e}t}(S,\mathbb{Z}/p\mathbb{Z})=0$. Using Leray's spectral sequence, we conclude  that $H^2(G,\mathbb{Z}/p\mathbb{Z})=0$. Similarly, $H^2(H,\mathbb{Z}/p\mathbb{Z})=0$ for any open subgroup $H$ of $G$. The group $G$ is therefore of cohomological $p$-dimension $\leq 1$, and it remains to apply lemma \ref{pdimcoh}.
\end{proof}

\begin{prop}\label{curvecyclotomic}
Let  $S=C$ be a smooth curve over an algebraically closed field $F$. Denote by $G$ its fundamental group (at a given geometric point). Set $\mathcal T:=\Z_p(1)$, the usual Tate module if $p \neq 0 \in F$, or $\mathcal T:=\Z_p$, the trivial module if $p = 0 \in F$. Then,  the pair   $(G,\mathcal T)$ is  $(1,\infty)$-cyclotomic.
\end{prop}

\begin{proof}
We may assume that $C$ is connected. If $F$ has characteristic $p$, one can adapt the proof of  Proposition \ref{charpcyclotomic}. How to do it is obvious if $C$ is affine. If $C$ is proper, note that one still has $H^2_{\acute{e}t}(C,\mathbb{G}_a)=0$, and that $(\Frob-\Id)$ induces a surjection on the finite-dimensional  $F$-vector space $H^1_{\acute{e}t}(C,\mathbb{G}_a)$ (see \cite[Lemma 0.5]{Bh}). A similar proof then goes through.\\
We thus assume that $F$ has characteristic not  $p$, and work over the small \'etale site over $C$. First assume that $C$ is proper and consider a connected finite \'etale cover $U\lra C$, given by an open subgroup $H$ of $G$. The curve $U$ is then smooth and proper, as well. Write the short exact sequence
$$0\lra \mathrm{Pic}^0(U)\lra \mathrm{Pic}(U)\stackrel{\mathrm{deg}}{\lra }\mathbb{Z}\lra 0.$$
The abelian group $\mathrm{Pic}^0(U)$ consists of the $F$-rational points of the Jacobian $\mathrm{Jac}_F(U)$. It is hence  divisible, since $F$ is algebraically closed.

Note that for any $s\geq 1$, the natural map $H^1_{\acute{e}t}(U,\mu_{p^s,U})\lra H^1_{\acute{e}t}(U,\mathbb{G}_{m,U})$ lands in $\mathrm{Pic}^0(U)$. Considering the same diagram as in the proof of Proposition \ref{semiloccyclotomic}, we see that the image of any class of $H^1_{\acute{e}t}(U,\mu_{p^s,U})$ in $H^1_{\acute{e}t}(U,\mathbb{G}_{m,U})$ lies in the image of the endomorphism of $H^1_{\acute{e}t}(U,\mathbb{G}_{m,U})$ induced by $\phi'$ (which is multiplication by $p$). The map $$H^1(H, \Z/p^{s+1} \Z (1))= H^1_{\acute{e}t}(U,\mu_{p^{s+1},U})\lra H^1_{\acute{e}t}(U,\mu_{p^{s},U})= H^1(H, \Z/p^{s} \Z (1))$$ induced by $\phi$ is thus also surjective. Therefore, the group $G$ is $(1,\infty)$-cyclotomic, with cyclotomic character the Tate module $\mathbb{Z}_p(1)$.\\
We now deal with the case of a non-proper (i.e. affine) smooth connected curve $C$ over $F$. Denote by $\tilde{C}$ the smooth proper curve containing $C$, and by $x$ a closed point in $\tilde{C}\setminus C$. Adjusting by multiples of $[x]\in \mathrm{Pic}(\tilde{C})$, one easily sees that the restriction morphism $\mathrm{Pic}^0(\tilde{C})\lra \mathrm{Pic}(C)$ is surjective, hence that the abelian group $\mathrm{Pic}(C)$ is divisible. The same holds for any \'etale cover of $C$, and we can conclude as before.
\end{proof}
The previous result actually extends to higher dimensional varieties, as follows.
\begin{prop}\label{vastcyclotomic}
 Let $S$ be a smooth projective variety, over an algebraically closed field $F$, of characteristic $\neq p$. Denote by $G$ the fundamental group of $S$, at  a given geometric point.\\
 Assume that, for every finite \'etale cover $U \lra S$, the N\'eron-Severi group of $U$ has no $p$-torsion. Then, the pair $(G,\Z_p(1))$ is $(1, \infty)$-cyclotomic. \\
 More generally, let $e \geq 1$ be an integer. Assume that, for every finite \'etale cover $U \lra S$, the $p$-primary part of the N\'eron-Severi group of $U$ is isomorphic to a product $\prod_{i=1}^n \Z /p^{r_i} \Z,$  with $r_i > e$ for $i=1,\ldots,n$. \\Then, the pair $(G,\Z/p^{1+e}(1))$ is $(1,e)$-cyclotomic.
\end{prop}
\begin{dem}
It suffices to prove the second statement. Consider the exact sequence (of Abelian groups) \[ 0 \lra \mathrm{Pic}^0(U) \lra \mathrm{Pic}(U) \lra \mathrm{NS}(U) \lra 0.\] The group $\mathrm{Pic}^0(U)$ is divisible (as the group of $F$-points of an Abelian variety). From the hypothesis made on $\mathrm{NS}(U) $, we deduce that the map $$\begin{tabular}{ccc}
$\mathrm{Pic}(U)[p^{r+1}]$&$\lra$    &  $\mathrm{Pic}(U)[p]$ \\
     $x$ & $\mapsto$ & $p^r x$
\end{tabular}$$
is onto, using an elementary diagram chase left to the reader. This holds for every finite \'etale cover $U/S$. The rest of the proof is the same as in Proposition \ref{curvecyclotomic}.
\end{dem}

\begin{rem}
Construction of new examples of surfaces satisfying hypothesis of Proposition \ref{vastcyclotomic} is an interesting problem. A good starting point may be  Kodaira fibered surfaces.
\end{rem}

\section{Lifting representations of cyclotomic profinite groups}

The following is a reformulation of Definition \ref{wkgmod}, for free $(\W_r(k),G)$-modules.

\begin{defi}\label{cont}
Let $G$ be a profinite group, $p$ be a prime, $k$ be a field of characteristic $p$ and $r\geq 1$ be an integer. A representation $$G\lra \GL_d(\W_r(k))$$ is continuous if its kernel is open in $G$.\\A continuous representation $$\rho_{\infty}:G\lra \GL_{\infty}(\W(k))$$ is a compatible data, for all $r\in \mathbb{N}^{\ast}$, of  continuous representations $$\rho_r:G\lra \GL_r(\W_d(k)).$$ The compatiblity condition simply means that $\rho_{r+1}$ reduces to $\rho_{r},$  for all $r$.
\end{defi}

We now prove Theorem \ref{main},  providing a sufficient condition for  lifting  continuous representations of cyclotomic profinite groups.

\begin{thm}\label{main}
Let $k$ be a field of characteristic $p$, and $e\in \mathbb{N}^{\ast}\cup \{\infty\}$. Let $G$ be a $(1,e)$-cyclotomic profinite group relatively to $k$, and let $V_1$ be a $(k,G)$-module. \\Assume that there is an open subgroup $G_0\subset G$, of prime-to-$p$ index, two permutation $(k,G_0)$-modules $A$ and $B$, and a short exact sequence of $(k,G_0)$-modules
$$0\lra A \lra V_1 \lra B\lra 0.$$ Then, $V_1$ stably lifts to $p^{e+1}$-torsion coefficients.

Furthermore, $V_1$ itself lifts to $p^2$-torsion coefficients.
\end{thm}

\begin{proof}
We show the first statement of the theorem. The second then follows from Lemma \ref{ablift}. Let $\W_{1+e}(1)$ be a $(1,e)$-cyclotomic module, relatively to $k$ and $G$. We may replace $G_0$ by its intersection with the kernel of the character $G \lra k^\times$ giving the action of $G$ on $\W_{1}(k)(1)$, which has index prime-to-$p$ as well. We can thus assume that $\W_{1}(k)(1)   \simeq k$ has the trivial $G_0$-action. The Yoneda extension $$0\lra A \lra V_1 \lra B\lra 0,$$ corresponds to a cohomology class $$\mathcal{E}\in H^1(G_0,\mathrm{Hom}(B,A)).$$ Fixing two respective $G_0$-bases $Y$ and $X$ of $A$ and $B$, we have a $G_0$-equivariant isomorphism \[\mathrm{Hom}(B,A) \simeq \mathrm{Hom}(k^{X},k^{ Y}) \simeq k^{X\times Y}, \] through which the class $\mathcal{E}$ is given by an element of $H^1(G_0,k^{X\times Y})$. The $G_0$-set $X\times Y$ decomposes as a disjoint union $$X\times Y = \bigsqcup_{i\in \mathcal{I}}G_0/G_i$$
of $G_0$-orbits, where all $G_i$'s are open in $G_0$. Shapiro's lemma yields an isomorphism $$H^1(G_0,k^{X\times Y})\simeq \bigoplus_{i\in \mathcal{I}}H^1(G_i,k)$$
Deciphering the definition of a $(1,e)$-cyclotomic pair and the comparison lemma \ref{extcoh}, we get a surjection $$H^1(G_i,\W_{1+e}(1)) \longrightarrow H^1(G_i,k);$$  that is to say, $$\mathrm{YExt}^1_{(\W_{e+1}(k),G_i)}(\W_{e+1}(k),\W_{1+e}(1))\longrightarrow \mathrm{YExt}^1_{(k,G_i)}(k,k),$$ for all $i \in \mathcal I$. Hence, the natural map $$\mathrm{YExt}^1_{(\W_{e+1}(k),G_0)}(\W_{e+1}(k)^X,\W_{1+e}(1)^Y)\longrightarrow \mathrm{YExt}^1_{(k,G_0)}(k^X,k^Y)$$ is surjective. As a consequence, $V_1$ fits into a commutative diagram of $(\W_{e+1}(k),G_0)$-modules 
\begin{equation} \label{dia1}
\begin{split}
\xymatrix{
    0\ar[r]&\W_{1+e}(1)^Y\ar[r]\ar[d]& V_{e+1}\ar[r]\ar[d]&\W_{e+1}(k)^X\ar[r]\ar[d]&0\\
    0\ar[r]&k^Y\ar[r]&V_1\ar[r]&k^X\ar[r]&0,
  }
\end{split}
\end{equation}
 which serves as the definition of $V_{e+1}$.  Note that the vertical arrows are the natural reductions.\\
The cyclotomic module $\W_{1+e}(k)(1)$ is a free $\W_{e+1}(k)$-module. Hence so is $V_{e+1}$- yielding a stable lifting of $V_1$, to $p^{e+1}$-torsion coefficients.
\end{proof}


\section{Applications to Galois representations and local systems}

In this section we provide applications to Theorem \ref{main} to lifting  Galois representations and local systems.

\begin{thm}\label{t1}Let $k$ be a field of characteristic $p$, $e\in \mathbb{N}^{\ast}\cup\{\infty\}$. Let $G$ be a $(1,e)$-cyclotomic profinite group. Let $$\rho:G\rightarrow \GL_2(k)$$ be a continuous representation. Then, $\rho$ lifts to $p^2$-torsion coefficients.

Furthermore, $\rho$ stably lifts to $p^{e+1}$-torsion coefficients.

If $k=\mathbb{F}_2$, these results also hold for representations of $G$ of  dimension up to $4$.
\end{thm}

\begin{proof} Let $V_1$ be a $2$-dimensional $(k,G)$-module. There is a line $L_1 \subset V_1$ fixed by a pro-$p$-Sylow $G_p$ of $G$ \cite[8.3]{Se2}. The stabilizer $H=\mathrm{Stab}_{G_p}(L_1)$ is thus of prime-to-$p$ index in $G$, and we get a short exact sequence of $(k,H)$-modules 
$$0\lra L_1\lra V_1 \lra V_1/L_1 \lra 0.$$
As in the previous proofs, we can consider an open subgroup $G_0$ of $G$, of prime-to-$p$ index, for which the two characters giving the action on $L_1$ and  on $V_1/L_1$ are trivial. We can now apply Theorem \ref{main}.

We now assume that $k=\mathbb{F}_2$ and that $V_1$ is a $4$-dimensional $(\mathbb{F}_2,G)$-module. Again, there is a plane $P_1 \subset V_1$ stabilized by an open subgroup $G_0$ of $G$, of odd index. The continuous representation $V_1$ fits into a short exact sequence of $(\mathbb{F}_2,G_0)$-modules $$0\lra P_1\lra V_1 \lra V_1/P_1 \lra 0.$$ Replacing $G_0$ by an open subgroup of odd index, we can moreover assume that the $2$-dimensional $(\mathbb{F}_2,G_0)$-modules $P_1$ and $V_1/P_1$ both admit an $\F_2$-basis permuted by $G_0$. It remains, once more, to invoke Theorem \ref{main}.
\end{proof}

\begin{rem}
Theorem \ref{t1} applies, in particular, to profinite groups of the types a), b), c) and d) given in the Introduction. In particular, it applies to Galois representations. To conclude, we offer another application below.
\end{rem}

\begin{coro}[Zariski-local lifting of local systems of low dimension].\label{tf}\\
Let $k$ be a field of characteristic $p$ and let $S$ be a  scheme, defined either over $\F_p$ or $\Z[\frac 1 p]$. Let $\mathcal L$ be a local system over $S$ with coefficients in $k$, of dimension $2$. \\  
(Equivalently, $\mathcal L$ is given by a representation $\pi_1(S) \lra \GL_2(k)$)\\
Then, Zariski-locally on $S$, $\mathcal L$  lifts to  a local system with coefficients in $\W_2(k)$.\\ Furthermore,  Zariski-locally on $S$, $\mathcal L$ stably lifts to a local system with coefficients in $\W(k)$.

If $k=\mathbb{F}_2$, the same result holds, for local systems of dimension up to  $4$.
\end{coro}

\begin{proof}
By Proposition \ref{semiloccyclotomic}, combined with Theorem \ref{t1}, we know that, for each point $s \in S$, the stalk of $\mathcal L$ at $s$ (which is a local system over $\Spec(\mathcal O_{S,s}$)) lifts as stated.  To conclude, use the fact that any finite \'etale cover of $\Spec(\mathcal O_{S,s})$ extends to an open $U \subset S$ containing $s$.
\end{proof}

\begin{rem}

We did not attempt to make the assumptions on $S$ optimal. For instance, the result clearly extends to schemes $S$ where $p$ is nilpotent. 

\end{rem}


\section*{Acknowledgments} We thank the referee for her/his careful reading and helpful suggestions. The idea of considering fundamental groups of smooth curves over an algebraically closed field as cyclotomic profinite occured during an enjoyable discussion with Adam Topaz, a while ago.


\begin{thebibliography}{2}
\bibitem[Bh]{Bh} B. Bhatt, Non-liftability of vector bundles to the Witt vectors. Available on the author webpage \url{http://www-personal.umich.edu/~bhattb/math/witt-vector-bundle-example.pdf}\\

\bibitem[DCF]{DCF} C. De Clercq and M. Florence, Smooth profinite groups and applications. Available on the arXiv at \url{https://arxiv.org/abs/1710.10631}, 2018.\\

\bibitem[DCF1]{DCF1} C. De Clercq and M. Florence, Smooth profinite groups I, geometrizing Kummer theory. Available on the arXiv at \url{https://arxiv.org/abs/2009.11130}, 2021.\\

\bibitem[DCF3]{DCF3}  C. De Clercq and M. Florence, Smooth profinite groups III, the Smoothness Theorem. Available on the arXiv at \url{https://arxiv.org/abs/2012.11027}, 2021.\\

\bibitem[F2]{F2}  M. Florence, Smooth profinite groups II, the Uplifting Theorem. Available on the arXiv at \url{https://arxiv.org/abs/2009.11140}, 2021.\\

\bibitem[Gi]{Gi} P. Gille, Le groupe fondamental sauvage d'une courbe affine  en   caractéristique $p>0$, Courbes semi-stables et groupe  fondamental en g\'eom\'etrie alg\'ebrique (Luminy, 1998), 217-231,  Progr. Math., 187, Birkh\"auser, Basel, 2000.  \\

\bibitem[K]{K} C. Khare, Base Change, Lifting and Serre’s Conjecture. J. Number Theory 63, 387–-395, 1997.\\

\bibitem [KL]{KL} \textsc{C. B. Khare, M. Larsen}, \textit{Liftable groups, negligible cohomology and Heisenberg representations,}  Available on the arXiv at \url{https://arxiv.org/abs/2009.01301}, 2020. \\

\bibitem[QW]{QW} C. Quadrelli, T. Weigel, Profinite groups with a cyclotomic $p$-orientation, to appear in Doc. Math.\\

\bibitem[R]{R} R. Ramakrishna, Lifting Galois representations, Invent. Math. 138 no. 3, 537-–562, 1999.\\

\bibitem[Se]{Se} J.-P. Serre, Galois cohomology, Springer-Verlag, Berlin, 1997.\\

\bibitem[Se2]{Se2} J.-P. Serre, Repr\'esentations lin\'eaires des groupes finis, 5\`eme \'edition, Hermann, 1998.\\

\bibitem[Se3]{Se3} J.-P. Serre, Sur la topologie des variétés algébriques en charactéristique $p$, in Symposium Internacional de Topologia Algebraica, Univ. Nac. Aut. de Mexico, 24-–53,  1958.\\

\bibitem[Ve]{Ve} J.-L. Verdier, Des cat\'egories d\'eriv\'ees des cat\'egories ab\'eliennes,  Ast\'erisque 239, Société Mathématique de France, 1996. \\
\end{thebibliography}
\end{document}